\documentclass[a4paper,12pt]{amsart}
\usepackage[margin=2.5cm]{geometry}
\usepackage[utf8]{inputenc}
\usepackage[T1]{fontenc}
\usepackage[charter]{mathdesign}
\usepackage{hyperref}
\hypersetup{
bookmarks=true,
pdfpagemode=UseNone,
pdfstartview=FitH,
pdfdisplaydoctitle=true,
pdfborder={0 0 0}, 
pdftitle={Cube spaces and the multiple term return times theorem},
pdfauthor={Pavel Zorin-Kranich},
pdfsubject={Mathematics Subject Classification (2010): 28D05 (Primary), 37A05 (Secondary)},
pdfkeywords={cube space, return times theorem, Wiener-Wintner theorem},
pdflang=en-US
}
\usepackage[msc-links,initials,backrefs,alphabetic]{amsrefs} 

\numberwithin{equation}{section}
\newtheorem{theorem}[equation]{Theorem}
\newtheorem{lemma}[equation]{Lemma}
\newtheorem{proposition}[equation]{Proposition}
\newtheorem{corollary}[equation]{Corollary}
\theoremstyle{definition}
\newtheorem{definition}[equation]{Definition}
\theoremstyle{remark}

\newcommand*{\N}{\mathbb{N}}
\newcommand*{\Z}{\mathbb{Z}}

\newcommand*{\Q}{\mathbb{Q}}

\newcommand*{\dif}{\mathrm{d}}
\newcommand*{\m}{\mathrm{m}}
\def\<{\left\langle}
\def\>{\right\rangle}
\newcommand*{\inv}{^{-1}}
\newcommand*{\id}{\mathrm{id}}

\newcommand*{\E}{\mathbb{E}}
\newcommand{\aveN}{\frac{1}{N}\sum_{n=1}^N}

\newcommand{\aveFN}{\frac{1}{|F_N|}\sum_{n\in F_N}}
\newcommand{\aveFMn}{\frac{1}{|F_M|}\sum_{n\in F_M}}

\newcommand*{\HKZ}{\mathcal{Z}}
\newcommand*{\RTT}[1]{$\mathrm{RTT}(#1)$}

\begin{document}
\subjclass[2010]{28D05 (Primary), 37A05 (Secondary)}
\title{Cube spaces and the multiple term return times theorem}
\author{Pavel Zorin-Kranich}
\address
{Korteweg-de Vries Institute for Mathematics\\
University of Amsterdam\\
P.O.\ Box 94248\\
1090 GE Amsterdam\\
The Netherlands}
\email{zorin-kranich@uva.nl}
\urladdr{http://staff.science.uva.nl/~pavelz/}
\keywords{Cube space, return times theorem, Wiener-Wintner theorem}
\begin{abstract}
We give a new proof of Rudolph's multiple term return times theorem based on Host-Kra structure theory.
Our approach provides characteristic factors for all terms, works for arbitrary tempered F\o{}lner sequences and also yields a multiple term Wiener-Wintner-type return times theorem for nilsequences.
\end{abstract}
\maketitle
\allowdisplaybreaks

\section{Introduction}
In this article we are concerned with universally good weights for pointwise convergence of ergodic averages along a tempered F\o{}lner sequence $(F_{N})$ in $\Z$, i.e., sequences $(a_{n})$ such that, for every measure-preserving system $(Y,S)$ and every $g\in L^{\infty}(Y)$, the averages
\[
\lim_{N} \aveFN a_{n} g(S^{n}y)
\]
converge for a.e.\ $y\in Y$.
Bourgain's return times theorem \cite{MR1557098} asserts that, given any ergodic measure-preserving system $(X,T)$, for every $f\in L^{\infty}(X)$ and a.e.\ $x\in X$ the sequence of weights $a_{n}=f(T^{n}x)$ is universally good for pointwise convergence along the standard F\o{}lner sequence $F_{N}=[1,N]$.
The name ``return times theorem'' comes from the case of a characteristic function $f=1_{A}$, $A\subset X$.
Then the theorem can be equivalently formulated by saying that, for a.e.\ $x\in X$, the pointwise ergodic theorem on any system $Y$ holds along the sequence of return times of $x$ to $A$.

This has been extended to averages involving multiple terms by Rudolph \cite{MR1489899}.
In order to formulate his result and for future convenience we now introduce some notation.
By a \emph{system} we mean an ergodic regular measure-preserving system $(X,\mu,T)$ with a distinguished countable subset $D\subset L^{\infty}(X)$ that is sufficiently large (we will formulate the precise condition on $D$ in Definition~\ref{def:D}).
\begin{definition}
\label{def:uae}
Let $P$ be a statement about ergodic regular measure-preserving systems $(X_{i},\mu_{i},T_{i})$, functions $f_{i}\in L^{\infty}(X_{i})$ and points $x_{i}\in X_{i}$, $i=0,\dots,k$.
We say that $P$ holds for \emph{universally almost every (u.a.e.)} tuple $x_{0},\dots,x_{k}$ if
\begin{itemize}
\item[$(0)$] For every system $(X_{0},\mu_{0},T_{0},D_{0})$ there exists a set of full measure $\tilde X_{0}\subset X_{0}$ such that
\item[$(1)$] for every system $(X_{1},\mu_{1},T_{1},D_{1})$ there exists a measurable set $\tilde X_{1}\subset X_{0}\times X_{1}$ such that for every $\vec x_{0}\in\tilde X_{0}$ the set $\{x_{1}: (\vec x_{0},x_{1})\in\tilde X_{1}\}$ has full measure in $X_{1}$ and
\item[]\begin{center}$\vdots$\end{center}
\item[$(k)$] for every system $(X_{k},\mu_{k},T_{k},D_{k})$ there exists a measurable set $\tilde X_{k}\subset X_{0}\times\dots\times X_{k}$ such that for every $\vec x_{k-1}\in\tilde X_{k-1}$ the set $\{x_{k}: (\vec x_{k-1},x_{k})\in\tilde X_{k}\}$ has full measure in $X_{k}$ and
\end{itemize}
we have $P(f_{0},\dots,f_{k},\vec x_{k})$ for every $\vec x_{k}\in\tilde X_{k}$ and every $f_{i}\in D_{i}$, $i=0,\dots,k$.
\end{definition}
With this convention Rudolph's multiple term return times theorem says that for every $k\in\N$ the averages
\begin{equation}
\label{eq:av}
\aveN f_{0}(T_{0}^{n}x_{0})\cdots f_{k}(T_{k}^{n}x_{k})
\end{equation}
converge for u.a.e.\ $x_{0},\dots,x_{k}$.
We call this statement \RTT{k}.
Birkhoff's pointwise ergodic theorem \cite{0003.25602} is essentially \RTT{0} and Bourgain's return times theorem is \RTT{1}.
More about the history of these and related results can be found in a recent survey by Assani and Presser \cite{2012arXiv1209.0856A}.

It is known that the Host-Kra-Ziegler pro-nilfactor $\HKZ_{k}(X_{0})$ is characteristic for the first term in \RTT{k} in the sense that if $f_{0}\perp\HKZ_{k}(X_{0})$, then the averages \eqref{eq:av} converge to zero u.a.e.\ \cite{MR2901351}*{Theorem 4}.
However, the proof of this fact hitherto depends on the convergence result \RTT{k}.

In this article we prove both results, \RTT{k} and characteristicity, simultaneously by induction on $k$ using the Host-Kra structure theory and extend them to arbitrary tempered F\o{}lner sequences.
We also obtain the following Wiener-Wintner return times theorem for nilsequences thereby generalizing \cite{MR1357765}*{Theorem 1}.
\begin{theorem}[Wiener-Wintner return times theorem for nilsequences]
\label{thm:WWRTT}
Let $k,l\in\N$ and $f_{i}\in L^{\infty}(X_{i})$, $i=0,\dots,k$.
Then for u.a.e. $x_{0},\dots,x_{k}$ and every $l$-step nilsequence $(a_{n})_{n}$ the averages
\[
\aveFN a_{n} \prod_{i=0}^{k}f_{i}(T_{i}^{n}x_{i})
\]
converge (to zero if in addition $f_{0}\perp\HKZ_{k+l}(X_{0})$ or $f_{i}\perp\HKZ_{k+l+1-i}(X_{i})$ for some $i=1,\dots,k$).
\end{theorem}
The general strategy is to consider not only the $X_{i}$'s but also the Host-Kra cube spaces of all orders simultaneously and to use a version of the convergence criterion due to Bourgain, Furstenberg, Katznelson and Ornstein (Proposition~\ref{prop:BFKO}) to pass from $k$ to $k+1$.
After a preparatory Section~\ref{sec:tools} we formulate our central convergence Theorem~\ref{thm:RTT-cube} on cube spaces, generalizing \RTT{k} and giving information about characteristic factors.
The remaining part of Section~\ref{sec:RTT-cube} is devoted to the proof of Theorem~\ref{thm:RTT-cube}.
Theorem~\ref{thm:WWRTT} is then proved in Section~\ref{sec:ww}.

\section{Notation and tools}
\label{sec:tools}
\subsection{Tempered F\o{}lner sequences}
A sequence $(F_{N})$ of finite subsets of $\Z$ is called a \emph{F\o{}lner sequence} if for every $k\in\Z$ we have $|F_{N} \Delta (F_{N}+k)|/|F_{N}| \to 0$ as $N\to\infty$.
A F\o{}lner sequence is called \emph{tempered} if there exists a constant $C$ such that for all $N$ one has
\[
\big| \cup_{M<N} F_{N}-F_{M} \big| \leq C |F_{N}|.
\]
Throughout the article we fix a tempered F\o{}lner sequence $(F_{N})$.

Let $(X,\mu,T)$ be an ergodic measure-preserving system and $f\in L^{\infty}(X)$.
A point $x\in X$ is called \emph{($\mu$-)generic} for $f$ if
\[
\aveFN f(T^{n}x) \to \int f \dif\mu
\quad
\text{as } N\to\infty.
\]
A point $x\in X$ is called \emph{fully ($\mu$-)generic} for $f$ if it is ($\mu$-)generic for every function in the closed $T$-invariant algebra spanned by $f$.
By the Lindenstrauss pointwise ergodic theorem \cite{MR1865397}, for every $f\in L^{1}(X)$ a.e.\ $x\in X$ is generic.
Consequently, for every $f\in L^{\infty}(X)$ a.e.\ $x\in X$ is fully generic.

\subsection{Ergodic decomposition}
For the purposes of this article we find it illuminating to think of the ergodic decomposition in a particular way (that will be generalized in Section~\ref{sec:RTT-cube}).
Let $(X,\mu,T)$ be a regular ergodic measure-preserving system, i.e.\ $X$ is a compact metric space, $T:X\to X$ is an invertible continuous map and $\mu$ is a $T$-invariant ergodic Borel probability measure.
By the pointwise ergodic theorem a.e.\ $x\in X$ is generic for some $T$-invariant Borel probability measure $\m_{x}$ on $X$, i.e.\ $\aveFN f(T^{n}x) \to \int f \dif \m_{x}$ for every $f\in C(X)$.
It follows easily that the function $x\mapsto\m_{x}$ is measurable and
\begin{equation}
\label{eq:m-disint}
\mu = \int\m_{x} \dif\mu(x)
\end{equation}
In particular, for $\mu$-a.e.\ $x$ the measure $\m_{y}$ is defined for $\m_{x}$-a.e.\ $y$.
To see that $\m_{x}$ is ergodic for $\mu$-a.e.\ $x$ it suffices to verify that
\begin{equation}
\label{eq:my-mx}
\int \int \Big| \int f \dif\m_{y} - \int f \dif\m_{x} \Big|^{2} \dif\m_{x}(y) \dif\mu(x)
= 0 \text{ for every } f\in C(X),
\end{equation}
since this says precisely that the ergodic averages of $f$ converge pointwise $\m_{x}$-a.e.\ to an $\m_{x}$-essentially constant function for $\mu$-a.e.\ $x$, and the latter full measure set can be chosen independently from $f$ since $C(X)$ is separable.
By definition of $\m_{x},\m_{y}$, the dominated convergence theorem and \eqref{eq:m-disint} we can rewrite the integral in \eqref{eq:my-mx} as
\begin{multline*}
2 \lim_{N} \int (\aveFN T^{n}f)^{2}(x) \dif\mu(x)
- 2 \lim_{N} \int (\aveFN T^{n}f)(x) \int (\aveFN T^{n}f)(y) \dif\m_{x}(y) \dif\mu(x)\\
=
2 \lim_{N} \int (\aveFN T^{n}f)^{2}(x) \dif\mu(x)
- 2 \lim_{N} \lim_{M} \int (\aveFN T^{n}f)(x) (\aveFMn T^{m}f)(x) \dif\mu(x),
\end{multline*}
and this vanishes by the pointwise ergodic theorem and the dominated convergence theorem.

\subsection{Nilsystems}
A \emph{($k$-step) nilmanifold} is a homogeneous space $G/\Gamma$, where $G$ is a ($k$-step) nilpotent Lie group and $\Gamma$ is a discrete cocompact subgroup.
A nilmanifold is always implicitly endowed with the Haar measure, the unique left-$G$-invariant Borel probability measure.
A \emph{($k$-step) nilsystem} is a measure-preserving system of the form $(X,T)$, where $X=G/\Gamma$ is a ($k$-step) nilmanifold and $Tg\Gamma = ag\Gamma$ for some $a\in G$ and all $g\Gamma\in G/\Gamma$.
A \emph{basic ($k$-step) nilsequence} is a sequence of the form $a_{n}=F(T^{n}x)$, where $(X,T)$ is a ($k$-step) nilsystem, $x\in X$ and $F\in C(X)$ is a continuous function.
A \emph{$k$-step nilsequence} is a uniform limit of basic $k$-step nilsequences (it would be more consistent to call basic nilsequences ``nilsequences'' and nilsequences ``pro-nilsequences'' but we follow the established terminology).
A \emph{$k$-step pro-nilsystem} is an inverse limit of $k$-step nilsystems in the category of measure-preserving systems (equivalently, in the category of topological dynamical systems with an invariant Borel probability measure \cite{MR2600993}*{Theorem A.1}).
A \emph{(pro-)nilfactor} of a measure-preserving dynamical system is a factor that is also a (pro-)nilsystem.

It is a classical fact that the Kronecker factor of an ergodic nilsystem $(G/\Gamma,T)$ is the canonical map $G/\Gamma \to G/\Gamma G_{2}$, where $G_{2}=[G,G]$.
The nilmanifold $G/\Gamma G_{2}$ is a compact homogeneous space of the abelian Lie group $G/G_{2}$, hence a disjoint union of finitely many tori.
The fibers of the projection $G/\Gamma \to G/\Gamma G_{2}$ are isomorphic to the homogeneous space $G_{2}/\Gamma_{2}$, where $\Gamma_{2}=\Gamma\cap G_{2}$.
By a result of Mal'cev $\Gamma_{2}$ is a cocompact subgroup of $G_{2}$ \cite{MR0028842}, so each such fiber is also a nilmanifold.

\subsection{Host-Kra structure theory}
We recall the basic definitions and main results surrounding the uniformity seminorms \cite{MR2150389}.
Let $(X,\mu,T)$ be a regular ergodic measure-preserving system.
The cube measures $\mu^{[l]}$ on $X^{[l]}:=X^{2^{l}}$ are defined inductively starting with $\mu^{[0]}:=\mu$.
In the inductive step, given $\mu^{[l]}$, fix an ergodic decomposition
\[
\mu^{[l]}=\int_{X^{[l]}} \m_{x} \dif\mu^{[l]}(x)
\]
as in \eqref{eq:m-disint}.
The space on which $\m_{x}$ is defined can be inferred from the subscript $x$.
Define
\begin{equation}
\label{eq:mul+1}
\mu^{[l+1]}:=\int_{X^{[l]}} \delta_{x}\otimes\m_{x} \dif\mu^{[l]}(x).
\end{equation}
To see that this coincides with the conventional definition one can use \eqref{eq:m-disint} and \eqref{eq:my-mx} to write the above integral as
\begin{multline}
\label{eq:mul+1-conventional}
\mu^{[l+1]}
=
\int\int \delta_{y}\otimes\m_{y} \dif\m_{x}(y)\dif\mu^{[l]}(x)\\
=
\int\int \delta_{y}\otimes\m_{x} \dif\m_{x}(y)\dif\mu^{[l]}(x)
=
\int \m_{x}\otimes\m_{x} \dif\mu^{[l]}(x).
\end{multline}
The \emph{uniformity seminorms} are
\begin{equation}
\label{eq:uniformity-seminorm-integral}
\| f \|_{U^{l+1}}^{2^{l+1}} := \int \otimes_{\epsilon\in \{0,1\}^{l+1}} f \dif\mu^{[l+1]} = \int \E\big( \otimes_{\epsilon\in \{0,1\}^{l}} f | \mathcal{I}^{[l]} \big)^{2} \dif\mu^{[l]},
\end{equation}
where $\mathcal{I}^{[l]}$ is the $T^{[l]}$-invariant sub-$\sigma$-algebra on $X^{[l]}$.
For $f^{\epsilon} \in L^{\infty}(X)$, $\epsilon\in\{0,1\}^{l}$, we will abbreviate $f^{[l]}:=\otimes_{\epsilon\in\{0,1\}^{l}}f^{\epsilon}$.
The uniformity seminorms satisfy the \emph{Cauchy-Schwarz-Gowers inequality} \cite{MR2150389}*{Lemma 3.9.(1)}
\begin{equation}
\label{eq:CSG}
\Big| \int f^{[l]} \dif\mu^{[l]} \Big| \leq \prod_{\epsilon\in\{0,1\}^{l}} \|f^{\epsilon}\|_{U^{l}}.
\end{equation}
The $U^{l+1}$-seminorms determine factors $\HKZ_{l}(X)$ by the relation
\[
f\perp L^{2}(\HKZ_{l}(X)) \iff \|f\|_{U^{l+1}}=0
\text{ for } f\in L^{\infty}(X).
\]
The main result of Host and Kra \cite{MR2150389} is that the factors $\HKZ_{l}(X)$ are $l$-step pro-nilsystems.

Our basic tool for proving convergence pointwise a.e.\ is the Wiener-Wintner theorem for nilsequences \cite{MR2544760}*{Theorem 2.22}.
We use the following version for tempered F\o{}lner sequences.
\begin{theorem}[{Wiener-Wintner for nilsequences \cite{arxiv:1208.3977}}]
\label{thm:WW}
Let $(X,T)$ be a regular measure-preserving system.
Then for every $f\in L^\infty(X)$ there exists a set $X'\subset X$ of full measure such that for every $x\in X'$ the averages
\[
\aveFN a_n f(T^n x)
\]
converge for every nilsequence $(a_n)$ as $N\to\infty$.
If in addition $f \perp \HKZ_{l}$ for some $l\in\N$ then the limit is zero for every $l$-step nilsequence and every $x\in X$ that is fully generic for $f$.
\end{theorem}

We also need the classical fact that the Kronecker factor is characteristic for $L^{2}$ convergence of ergodic averages with arbitrary bounded scalar weights, see e.g.\ \cite{MR2544760}*{Corollary 7.3} for a more general version.
\begin{lemma}
\label{lem:Z1-char-weight}
Let $(X,T)$ be an ergodic measure-preserving system and $f\in L^{2}(X)$ be orthogonal to $\HKZ_{1}(X)$.
Then for any bounded sequence $(a_{n})_{n}$ one has
\[
\lim_{N} \aveFN a_{n}T^{n}f = 0
\quad\text{in } L^{2}(X).
\]
\end{lemma}

\subsection{Conventions about cube measures}
In the sequel we will have to consider systems for which a certain approximating procedure can be carried out within their distinguished sets.
\begin{definition}
\label{def:D}
A \emph{system} is a regular ergodic measure-preserving system $(X,\mu,T)$ with a distinguished set $D\subset L^{\infty}(X)$ that satisfies the following conditions.
\begin{enumerate}
\item (Cardinality) $D$ is countable.
\item (Density) $D$ contains an $L^{\infty}$-dense subset of $C(X)$.
\item (Algebra) $D$ is a $\Q$-algebra and is closed under absolute value.
\item (Decomposition) For every $f\in D$ and $l\in\N$ there exist decompositions
\begin{equation}
\makeatletter
\def\tagform@#1{\maketag@@@{\ignorespaces#1\unskip\@@italiccorr} (l)}
\makeatother
\tag{Dec}
\label{eq:dec}
f=f_{\perp}+f_{\HKZ,j}+f_{err,j},
\quad j\in\N,
\end{equation}
such that $f_{\perp},f_{\HKZ,j},f_{err,j}\in D$, $f_{\perp}\perp\HKZ_{l}(X)$, $f_{\HKZ,j} \in C(Z_{j})$, where $Z_{j}$ is a nilfactor of $\HKZ_{l}(X)$, $\|f_{err,j}\|_{L^{\infty}(\mu)}$ is uniformly bounded in $j$ and $\|f_{err,j}\|_{L^{1}(\mu)}\to 0$ as $j\to\infty$.
\end{enumerate}
\end{definition}
For any regular ergodic measure-preserving system $(X,\mu,T)$ any countable subset of $L^{\infty}(X)$ is contained in a set $D$ that satisfies the above conditions.
Indeed, by the Host-Kra structure theorem \emph{every} bounded function on $X$ has a decomposition of the form \ref{eq:dec}$(l)$ for every $l\in\N$.

As a first preparatory step to the identification of universal sets in Theorem~\ref{thm:WWRTT} we choose well-behaved full measure sets from the cube spaces associated to the individual systems.
\begin{lemma}
\label{lem:Yl}
Let $(X,\mu,T,D)$ be a system.
Then there exist measurable subsets $Y_{l}\subset X^{[l]}$ such that for every $l\in\N$ the following statements hold.
\begin{enumerate}
\item $\mu^{[l]}(Y_{l}) = 1$ and for every $y\in Y_{l}$ we have $m_{y}(Y_{l})=1$.
\item\label{it:disintegration} For every $y\in Y_{l}$ the measure $\m_{y}$ is ergodic and one has
\begin{equation}
\label{eq:motimesm}
\m_{y}\otimes\m_{y} = \int_{Y_{l+1}} \m_{x} \dif(\m_{y}\otimes\m_{y})(x).
\end{equation}
\item\label{it:generic}
$Y_{l} \subset (\tilde X)^{[l]}$, where $\tilde X\subset X$ is the set of points that are generic for each $f\in D$ w.r.t.\ $\mu$.
\item\label{it:orth} For every $y\in Y_{l}$, every $k\in\N$ and any functions $f_{\epsilon}\in D$, $\epsilon\in\{0,1\}^{l}$, such that $f_{\epsilon}\perp\HKZ_{k+l}(X)$ for some $\epsilon$ we have $f^{[l]} \perp \HKZ_{k}(X^{[l]},\m_{y})$.
\end{enumerate}
\end{lemma}
\begin{proof}
The fact that (\ref{it:orth}) holds for full measure subsets of $X^{[l]}$ follows from the Cauchy-Schwarz-Gowers inequality~\eqref{eq:CSG}.
The sets $(\tilde X)^{[l]}\subset X^{[l]}$ have full measure by the pointwise ergodic theorem and the definition \eqref{eq:mul+1} of cube measures, taking care of (\ref{it:orth}).
Also, the measure $\m_{y}$ is ergodic for $\mu^{[l]}$-a.e.\ $y\in X^{[l]}$, taking care of the first part of (\ref{it:disintegration}).

The only delicate point is \eqref{eq:motimesm}.
By \eqref{eq:m-disint} and \eqref{eq:mul+1-conventional}, for a fixed full measure domain of integration this disintegration identity holds for $\mu^{[l]}$-a.e.\ $y\in X^{[l]}$.
However, the domain of integration is yet to be determined.
This is done by a fixed-point procedure: choose tentative sets $Y_{l} \subset X^{[l]}$ that satisfy all conditions but \eqref{eq:motimesm} for every $l$.
For every $l$ this gives a $\mu^{[l]}$-full measure subset of $y\in X^{[l]}$ for which \eqref{eq:motimesm} holds.
The intersection of this set with $Y_{l}$ gives a new tentative set $Y_{l}$.
This way for each $l$ we obtain a decreasing sequence of tentative full measure subsets of $X^{[l]}$ whose intersection still has full measure and satisfies all requested properties.
\end{proof}

\subsection{Bourgain-Furstenberg-Katznelson-Ornstein criterion}
Our main tool for proving convergence u.a.e.\ is the following criterion that reduces the search for a universal set of $x\in X$ (that a priori involves uncountably many systems $Y$) to a problem about $X^{2}$.
This is the step that necessitates the dependence of the universal sets in Definition~\ref{def:uae} on preceding systems.
\begin{proposition}
\label{prop:BFKO}
Let $(X,T)$ be an ergodic measure-preserving system and $f\in L^{\infty}(X) \cap \HKZ_{1}(X)^{\perp}$.
Assume that $x\in X$ is fully generic for $f$ and
\[
\aveFN f(T^{n}x) f(T^{n}\xi)\to 0
\quad\text{ for a.e. }\xi\in X.
\]
Then for every measure-preserving system $(Y,S)$ and $g\in L^{\infty}(Y)$ we have
\[
\aveFN f(T^{n}x)g(S^{n}y)\to 0
\quad\text{ for a.e. }y\in Y.
\]
\end{proposition}
Proposition~\ref{prop:BFKO} is due to Bourgain, Furstenberg, Katznelson and Ornstein in the case of the standard Ces\`aro averages \cite{MR1557098}*{Proposition}.
Their proof has been extended to F\o{}lner sequences in countable amenable groups satisfying the Tempelman condition by Ornstein and Weiss \cite{MR1195256}*{\textsection 3}.
Lindenstrauss \cite{MR1865397} observed that the result for tempered F\o{}lner sequences follows by methods of Ornstein and Weiss from his random covering lemma, for a detailed proof see \cite{arxiv:1301.1884}.

\subsection{A measure-theoretic lemma}
The next lemma is our main tool for dealing with cube measures.
Informally, it shows that a certain kind of universality for $\mu^{[1]}\otimes\nu^{[1]}$ implies some universality for $(\mu\times\nu)^{[1]}$.

Recall that, for ergodic measure-preserving systems $(X,\mu),(Y,\nu)$, the projection onto the invariant factor of $X\times Y$ has the form $\phi(x,y)=\psi(\pi_{1}(x),\pi_{1}(y))$, where $\pi_{1}$ are projections onto the Kronecker factors and $\psi$ is the quotient map of $\HKZ_{1}(X)\times\HKZ_{1}(Y)$ by the orbit closure of the identity.
To see this, recall that by Lemma~\ref{lem:Z1-char-weight} the function $f\otimes g$, $f\in L^{\infty}(X)$, $g\in L^{\infty}(Y)$, is orthogonal to the invariant factor of $X\times Y$ whenever $f\perp\HKZ_{1}(X)$ or $g\perp\HKZ_{1}(Y)$.
Thus the invariant sub-$\sigma$-algebra on $X\times Y$ is contained in $\HKZ_{1}(X)\times\HKZ_{1}(Y)$, i.e.\ it is (isomorphic to) the invariant sub-$\sigma$-algebra of a product of two compact group rotations (cf.\ e.g.\ \cite{MR1325712}*{Theorem 1.9}).
In particular, for an ergodic system $Y$ the invariant factor of $Y\times Y$ is isomorphic to $\HKZ_{1}(Y)$.
\begin{lemma}
\label{lem:cube-change-order}
Let $(X,\mu),(Y,\nu)$ be ergodic measure-preserving systems and fix measure disintegrations
\[
\mu = \int_{\kappa\in\HKZ_{1}(X)} \mu_{\kappa} \dif\kappa,
\quad
\nu = \int_{\lambda\in\HKZ_{1}(Y)} \mu_{\lambda} \dif\lambda.
\]
This induces an ergodic decomposition
\[
\nu\otimes\nu = \int_{\lambda\in\HKZ_{1}(Y)} (\nu\otimes\nu)_{\lambda} \dif\lambda,
\quad
(\nu\otimes\nu)_{\lambda} = \int_{\lambda'\in\HKZ_{1}(Y)} \nu_{\lambda'}\otimes\nu_{\lambda' \lambda\inv} \dif\lambda'.
\]
Let $x\in X$ and $\Lambda\subset\HKZ_{1}(Y)$ be a full measure set.
Assume that for $\mu$-a.e. $\xi$ and every $\lambda\in\Lambda$, for $(\nu\otimes\nu)_{\lambda}$-a.e. $(\eta,\eta')$ one some statement $P(x,\xi,\eta,\eta')$ holds.
Then $P(x,\xi,y,\eta)$ also holds for $\nu$-a.e. $y$ and $\tilde\m_{x,y}$-a.e. $(\xi,\eta)$, where
\[
\tilde\m_{x,y} =
\int_{\kappa\in\HKZ_{1}(X),\lambda\in\HKZ_{1}(Y):\psi(\pi_{1}(x),\pi_{1}(y))=\psi(\kappa,\lambda)} \mu_{\kappa}\otimes\nu_{\lambda} \dif(\kappa,\lambda),
\]
the homomorphism $\psi$ is as above and the integral is taken over an affine subgroup (i.e.\ a coset of a closed subgroup) with respect to its Haar measure.
\end{lemma}
\begin{proof}
Recall that $\ker\psi$ has full projections on both coordinates.
Therefore, for \emph{every} $x$ there is a full measure set of $\xi$ such that the set $\Lambda$ has full measure in $\{\lambda : \psi(\pi_{1}(x)\pi_{1}(\xi)\inv,\lambda)=\id\}$ (note that this is a closed affine subgroup of $\HKZ_{1}(Y)$ that therefore has a Haar measure).

In particular, for a full measure set of $\xi$ (that depends on $Y$) the hypothesis holds for a.e. $\lambda$ with $\psi(\pi_{1}(x)\pi_{1}(\xi)\inv,\lambda)=\id$, i.e.\ we have $P(x,\cdot)$ for a set of full measure w.r.t.\ the measure
\begin{multline*}
\int_{\xi\in X} \delta_{\xi}\otimes \int_{{\lambda\in\HKZ_{1}(Y) : \psi(\pi_{1}(x)\pi_{1}(\xi)\inv,\lambda)=\id}}
(\nu\otimes\nu)_{\lambda} \dif\lambda \dif\mu(\xi)\\
=
\int_{\kappa\in\HKZ_{1}(X)} \int_{\lambda\in\HKZ_{1}(Y):\psi(\pi_{1}(x)\kappa\inv,\lambda)=\id} \mu_{\kappa}\otimes
(\nu\otimes\nu)_{\lambda} \dif\lambda \dif\kappa\\
=
\int_{\kappa\in\HKZ_{1}(X)} \int_{\lambda\in\HKZ_{1}(Y):\psi(\pi_{1}(x)\kappa\inv,\lambda)=\id} \mu_{\kappa}\otimes
\int_{\lambda'\in\HKZ_{1}(Y)} \nu_{\lambda'}\otimes\nu_{\lambda'\lambda\inv} \dif\lambda' \dif\lambda \dif\kappa\\
=
\int_{\kappa\in\HKZ_{1}(X)} \int_{\lambda\in\HKZ_{1}(Y):\psi(\pi_{1}(x)\kappa\inv,\lambda)=\id} \mu_{\kappa}\otimes
\int_{y\in Y} \delta_{y}\otimes\nu_{\pi(y)\lambda\inv} \dif\nu(y) \dif\lambda \dif\kappa\\
=
\int_{y\in Y}
\int_{\kappa\in\HKZ_{1}(X)} \int_{\lambda\in\HKZ_{1}(Y):\psi(\pi_{1}(x)\kappa\inv,\lambda)=\id} \mu_{\kappa}\otimes
\delta_{y}\otimes\nu_{\pi(y)\lambda\inv} \dif\lambda \dif\kappa \dif\nu(y)\\
=
\int_{y\in Y} \int_{\kappa\in\HKZ_{1}(X),\lambda\in\HKZ_{1}(Y):\psi(\pi_{1}(x),\pi_{1}(y))=\psi(\kappa,\lambda)} \mu_{\kappa}
\otimes \delta_{y}\otimes\nu_{\lambda} \dif(\kappa,\lambda) \dif\nu(y)\\
=
\int_{y\in Y} \delta_{y}\otimes \tilde\m_{x,y} \dif\nu(y).
\end{multline*}
This gives $P(x,\xi,y,\eta)$ for $\nu$-a.e. $y$ and $\tilde\m_{x,y}$-a.e. pair $(\xi,\eta)$ as required.
\end{proof}
The next lemma provides us with means for using the measure $\tilde\m_{x,y}$ in a higher step setting.
\begin{lemma}
\label{lem:prod-fiber-erg}
Let $(Z,g),(Z',g')$ be ergodic nilsystems and $\psi:\HKZ_{1}(Z)\times\HKZ_{1}(Z') \to H$ the factor map modulo the orbit closure of $(\pi_{1}(g),\pi_{1}(g'))$.
Then for every $\lambda\in\HKZ_{1}(Z)$ and a.e.\ $\lambda'\in\HKZ_{1}(Z')$ the rotation by $(g,g')$ on the nilmanifold
\[
N_{\lambda,\lambda'} = \{(z,z')\in Z\times Z' : \psi(\pi_{1}(z),\pi_{1}(z')) = \psi(\lambda,\lambda')\}
\]
is uniquely ergodic.
\end{lemma}
\begin{proof}
By \cite{MR2122919}*{2.17-2.20} it suffices to prove ergodicity to obtain unique ergodicity.

Since $N_{\lambda,\lambda'}$ only depends on $\psi(\lambda,\lambda')$ and $\ker\psi$ has full projection on $\HKZ_{1}(Z)$ it suffices to verify the conclusion for a full measure set of $(\lambda,\lambda')$.
For this end it suffices to check that for any $f\in C(Z),f'\in C(Z')$ the limit of the ergodic averages of $f\otimes f'$ is essentially constant on $N_{\lambda,\lambda'}$.
We decompose $f=f_{\perp}+f_{\HKZ}$ with $f_{\perp}\perp\HKZ_{1}(Z)$ and $f_{\HKZ}\in L^{\infty}(\HKZ_{1}(Z))$, and analogously for $f'$.
For $f_{\HKZ}\otimes f'_{\HKZ}$ the limit is essentially constant on $N_{\lambda,\lambda'}$ for any $(\lambda,\lambda')$ since the rotation is ergodic on $(\pi_{1}\times\pi_{1})(N_{\lambda,\lambda'})$.

On the other hand, the limit of the ergodic averages of tensor products involving $f_{\perp}$ vanishes on $Z\times Z'$ a.e.\ by Lemma~\ref{lem:Z1-char-weight}, hence also a.e.\ on a.e.\ fiber $N_{\lambda,\lambda'}$.
\end{proof}

\section{Return times theorem on cube spaces}
\label{sec:RTT-cube}
In order to concisely state our central result, Theorem~\ref{thm:RTT-cube}, we need a cube version of Definition~\ref{def:uae}.
Recall that we write $f_{i}^{[l]}=\otimes_{\epsilon\in\{0,1\}^{l}}f_{i,\epsilon}$, where $f_{i,\epsilon} \in L^{\infty}(X_{i})$.
\begin{definition}
\label{def:luae}
Let $P$ be a statement about ergodic regular measure-preserving systems $(X_{i},\mu_{i},T_{i})$, functions $f_{i}^{[l]}$ and points $x_{i}\in X_{i}^{[l]}$, $i=0,\dots,k$.
We say that $P$ holds for \emph{$[l]$-universally almost every ($[l]$-u.a.e.)} tuple $x_{0},\dots,x_{k}$ if
\begin{itemize}
\item[$(0)$] For every system $(X_{0},\mu_{0},T_{0},D_{0})$ there exists a measurable set $\tilde X_{0}^{[l]}\subset X_{0}^{[l]}$ such that for every $y_{0}\in Y_{0,l}$ we have $\m_{y_{0}}(\tilde X_{0}^{[l]})=1$ and
\item[$(1)$] for every system $(X_{1},\mu_{1},T_{1},D_{1})$ there exists a measurable set $\tilde X_{1}^{[l]}\subset X_{0}^{[l]}\times X_{1}^{[l]}$ such that for every $\vec x_{0}\in\tilde X_{0}^{[l]}$ and every $y_{1}\in Y_{1,l}$ we have $\m_{y_{1}}\{x_{1}: (\vec x_{0},x_{1})\in\tilde X_{1}\}=1$ and
\item[]\begin{center}$\vdots$\end{center}
\item[$(k)$] for every system $(X_{k},\mu_{k},T_{k},D_{k})$ there exists a measurable set $\tilde X_{k}^{[l]}\subset X_{0}^{[l]}\times\dots\times X_{k}^{[l]}$ such that for every $\vec x_{k-1}\in\tilde X_{k-1}^{[l]}$ and every $y_{k}\in Y_{k,l}$ we have $\m_{y_{k}}\{x_{k}: (\vec x_{k-1},x_{k})\in\tilde X_{k}\}=1$ and
\end{itemize}
we have $P(f_{0}^{[l]},\dots,f_{k}^{[l]},\vec x_{k})$ for every $\vec x_{k}\in\tilde X_{k}$ and any $f_{i,\epsilon}\in D_{i}$, $0\leq i\leq k$, $\epsilon\in\{0,1\}^{l}$.
\end{definition}
With this definition ``u.a.e.''  corresponds to ``$[0]$-u.a.e.''.

Our main theorem below states that certain pro-nilfactors are characteristic for return time averages on cube spaces.
\begin{theorem}
\label{thm:RTT-cube}
For any $k,l\in\N$ the ergodic averages of $\otimes_{i=0}^{k}f_{i}^{[l]}$ converge $[l]$-u.a.e.
If in addition
\begin{equation}
\makeatletter
\def\tagform@#1{\maketag@@@{\ignorespaces#1\unskip\@@italiccorr} (k,l)}
\makeatother
\tag{CF}
\label{eq:cf}
\exists\epsilon\in\{0,1\}^{l}
\text{ s.t.\ }
f_{0,\epsilon}\perp\HKZ_{k+l}(X_{0})
\text{ or }
f_{i,\epsilon}\perp\HKZ_{k+l+1-i}(X_{i})
\text{ for some }
1\leq i\leq k
\end{equation}
then the limit vanishes $[l]$-u.a.e.
\end{theorem}
We refer to the statement of Theorem~\ref{thm:RTT-cube} with fixed $k,l$ as \RTT{k,l} and fixed $k$ and arbitrary $l$ as \RTT{k,\cdot}.
Note that \RTT{k,0} not only contains \RTT{k} but also gives additional information \ref{eq:cf}$(k,0)$ about characteristic factors.

We prove \RTT{k,\cdot} by induction on $k$.
The base case $k=0$ follows by definition of $Y_{0,l}$ and the pointwise ergodic theorem.

For the remaining part of this section we assume \RTT{k,\cdot} for some fixed $k$ and prove \RTT{k+1,\cdot}.
If $k>0$ then we also assume all other results of this section for $k-1$ in place of $k$ (thus, strictly speaking, it is the conjunction of all results in this section that is proved by induction).

In order to prove \RTT{k+1,l} for a given $l$ we write
\begin{equation}
\label{eq:X}
X_{0}^{[l+1]} \times\dots\times X_{k}^{[l+1]}
=
(X_{0}^{[l]} \times\dots\times X_{k}^{[l]})^{2}=:X^{2}.
\end{equation}
From \RTT{k,l+1} we know that the appropriate ergodic averages converge $[l+1]$-u.a.e.\ on $X^{2}$.
We would like to apply Proposition~\ref{prop:BFKO} with this $X$ and $Y=X_{k+1}^{[l]}$.
The remaining part of this section is dedicated to reformulating \RTT{k,l+1} in such a way that it can be plugged into Proposition~\ref{prop:BFKO}.

This involves the following steps.
First we use \RTT{k,\cdot} to construct a certain universal measure disintegration with built-in genericity properties on a product of ergodic systems (Theorem~\ref{thm:return-times-disintegration}).
We use characteristic factors for \RTT{k,\cdot} to represent measures in this disintegration in a different way.
Finally, we verify a certain instance of \RTT{k+1,\cdot} (Lemma~\ref{lem:cube-meas}).

\subsection{Universal disintegration of product measures}
The return times theorem can be seen as a statement about measure disintegration, cf.\ \cite{MR1357765}*{Theorem 4} for the special case $k=1$.
\begin{theorem}
\label{thm:return-times-disintegration}
Let $(X_{i},\mu_{i},T_{i},D_{i})$, $i=0,\dots,k$, be systems.
Then $[l]$-u.a.e.\ $x_{0},\dots,x_{k}$ is generic for some measure $\m_{x_{0},\dots,x_{k}}$ on $X_{0}^{[l]}\times\dots\times X_{k}^{[l]}$ and every function $\otimes_{i=0}^{k}f_{i}^{[l]}$, $f_{i,\epsilon}\in D_{i}$.

Moreover, for $[l]$-u.a.e. $x_{0},\dots,x_{k-1}$ and every $y_{k}\in Y_{l,k}$ one has
\begin{equation}
\label{eq:m-disintegration}
\m_{x_{0},\dots,x_{k-1}}\otimes\m_{y_{k}} = \int \m_{x_{0},\dots,x_{k}} \dif\m_{y_{k}}(x_{k}).
\end{equation}
\end{theorem}
\begin{proof}
By Theorem~\ref{thm:RTT-cube} with $l=0$ we obtain convergence of the averages
\[
\aveFN \prod_{i=0}^{k}f_{i}^{[l]}(T_{i}^{n}x_{i})
\]
for $[l]$-u.a.e.\ $x_{0},\dots,x_{k}$ and any $f_{i,\epsilon}\in D_{i}$.
For continuous functions $f_{i,\epsilon}\in D_{i}$ we define $\m_{x_{0},\dots,x_{k}}(\otimes_{i=0}^{k}f_{i}^{[l]})$ as the limit of these averages.
By the Stone-Weierstraß theorem these tensor products span a dense subspace $C(X_{0}^{[l]}\times\dots\times X_{k}^{[l]})$, so by density the above (bounded) linear form admits a unique continuous extension.

In order to obtain \eqref{eq:m-disintegration} it suffices to verify that the integrals of functions of the form $\otimes_{i=0}^{k} f_{i}^{[l]}$, $f_{i,\epsilon}\in D_{i}$, with respect to both measures coincide.
By genericity and the dominated convergence theorem we have for $[l]$-u.a.e.\ $x_{0},\dots,x_{k-1}$ that
\begin{multline*}
\int \int \otimes_{i<k}f_{i}^{[l]}\otimes f_{k}^{[l]} \dif\m_{x_{0},\dots,x_{k}} \dif\m_{y_{k}}(x_{k})\\
=
\int \lim_{N} \aveFN \prod_{i<k}f_{i}^{[l]}(T_{i}^{n}x_{i})\cdot f_{k}^{[l]}(T_{k}^{n}x_{k}) \dif\m_{y_{k}}(x_{k})\\
=
\lim_{N} \aveFN \prod_{i<k}f_{i}^{[l]}(T_{i}^{n}x_{i})\cdot \int f_{k}^{[l]}(T_{k}^{n}x_{k}) \dif\m_{y_{k}}(x_{k})\\
=
\int \otimes_{i<k}f_{i}^{[l]} \dif m_{x_{0},\dots,x_{k-1}} \int f_{k}^{[l]} \dif\m_{y_{k}}
\end{multline*}
as required.
\end{proof}
\subsection{Properties of the universal disintegration}
We will now represent the measure $\m_{x_{0},\dots,x_{k}}$ for $[l]$-u.a.e.\ $x_{0},\dots,x_{k}$ in the form $\tilde\m_{x,y}$ in the notation of Lemma~\ref{lem:cube-change-order}.
At this step we have to use the information about characteristic factors.
We begin with a preliminary observation.
\begin{lemma}
\label{lem:uae-fiber-ae}
If some property P holds for $[l]$-u.a.e.\ $x_{0},\dots,x_{k}$ then, for $[l]$-u.a.e.\ $x_{0},\dots,x_{k}$, P holds $\m_{x_{0},\dots,x_{k}}$-a.e.
\end{lemma}
\begin{proof}
For $k=0$ this follows from \eqref{eq:m-disintegration}.
Assume that the conclusion is known for $k-1$ and show it for $k$.

By the induction hypothesis, for $[l]$-u.a.e.\ $x_{0},\dots,x_{k-1}$, $\m_{x_{0},\dots,x_{k-1}}$-a.e., for every $y_{k}\in Y_{k,l}$, P holds $\m_{y_{k}}$-a.e.\ in $x_{k}$.
The conclusion follows from \eqref{eq:m-disintegration}.
\end{proof}
\begin{lemma}
\label{lem:m-tilde-m}
For $[l]$-u.a.e.\ $x_{0},\dots,x_{k}$ we have
\[
\m_{x_{0},\dots,x_{k}}=\tilde\m_{x,y},
\]
where we use the notation of Lemma~\ref{lem:cube-change-order} with
$(X,\mu)=(X_{0}^{[l]}\times\dots\times X_{k-1}^{[l]},\m_{x_{0},\dots,x_{k-1}})$, $(Y,\nu)=(X_{k}^{[l]},\m_{x_{k}})$, $x=(x_{0},\dots,x_{k-1})$ and $y=x_{k}$.
\end{lemma}
\begin{proof}
To verify that the measures coincide it suffices to check that the integrals of functions of the form $\otimes_{i=0}^{k}f_{i}^{[l]}$, $f_{i,\epsilon}\in D_{i}$ coincide.
For this end consider the splittings $f_{i,\epsilon}=f_{i,\epsilon,\perp}+f_{i,\epsilon,\HKZ,j}+f_{i,\epsilon,err,j}$, $j\in\N$, given by \ref{eq:dec}$(k+l+1-i)$.

Projections of tensor products that involve $f_{i,\epsilon,\perp}$ on one of the Kronecker factors vanish a.e. for $[l]$-u.a.e. $x,y$ by Corollary~\ref{cor:orth} for $k-1$ that is part of the induction hypothesis for this section.
Since $\ker\psi$ has full projections on both coordinates the corresponding integrals w.r.t. $\tilde\m_{x,y}$ also vanish.
The integrals w.r.t. $\m_{x,y}$ vanish for $[l]$-u.a.e. $x,y$ by Theorem~\ref{thm:RTT-cube}.

For the main terms we have
\begin{multline}
\label{eq:m-erg-main-term}
\int \otimes_{i=0}^{k} f_{i,\HKZ,j}^{[l]} \dif\tilde\m_{x,y}\\
=
\int_{\kappa\in\HKZ_{1}(X),\lambda\in\HKZ_{1}(Y):\psi(\pi_{1}(x),\pi_{1}(y))=\psi(\kappa,\lambda)} \E(\otimes_{i=0}^{k-1} f_{i,\HKZ,j}^{[l]}|\HKZ_{1}(X))(\kappa) \E(f_{k,\HKZ,j}^{[l]}|\HKZ_{1}(Y))(\lambda) \dif(\kappa,\lambda).
\end{multline}
Since the underlying nilmanifold of a nilsystem is a bundle of nilmanifolds over its Kronecker factor, the conditional expectation above is just integration in the fibers, and by uniqueness of the Haar measure the whole integral equals
\[
\int_{\kappa\in Z_{j},\lambda\in Z_{j}':\psi(\pi_{1}(x),\pi_{1}(y))=\psi(\pi_{1}(\kappa),\pi_{1}(\lambda))} \otimes_{i=0}^{k-1} f_{i,\HKZ,j}^{[l]}(\kappa) f_{k,\HKZ,j}^{[l]}(\lambda) \dif(\kappa,\lambda),
\]
where $Z_{j}$ is the orbit closure of $x$ in $\prod_{i=0}^{k-1}Z_{i,j}^{[l]}$ and $Z_{j}'$ is the orbit closure of $y$ in $Z_{k,j}^{[l]}$.
By Lemma~\ref{lem:prod-fiber-erg}, the above fibers of $Z_{j}\times Z_{j}'$ are uniquely ergodic for every $x$ and a.e. $y$, and the integral then equals
\[
\lim_{N} \aveFN \otimes_{i=0}^{k-1} f_{i,\HKZ,j}^{[l]}(T^{n}x) f_{k,\HKZ,j}^{[l]}(S^{n}y)
=
\int \otimes_{i=0}^{k}f_{i,\HKZ,j}^{[l]} \dif\m_{x,y}.
\]
It remains to treat the error terms, i.e. the case $f_{i',\epsilon'}=f_{i',\epsilon',err,j}$ for some $i',\epsilon'$.
By Lemma~\ref{lem:Yl}(\ref{it:generic}), for $[l]$-u.a.e. $x,y$ we have
\[
\int \otimes_{i=0}^{k}f_{i}^{[l]} \dif\m_{x,y} \lesssim \|f_{i',\epsilon'}\|_{L^{1}(\mu_{i})}
\to 0 \quad \text{as} \quad j\to\infty.
\]
Similarly, we have $\int | \otimes_{i=0}^{k-1}f_{i}^{[l]} | \dif\m_{x} \lesssim \|f_{i',\epsilon'}\|_{L^{1}(\mu_{i})}$ if $i'<k$ and $\int | f_{k}^{[l]} | \dif\m_{y_{k}} \lesssim \|f_{k,\epsilon}\|_{L^{1}(\mu_{i})}$ if $i'=k$ for $[l]$-u.a.e.\ $x,y$.
This implies that either $\E(\otimes_{i=0}^{k-1}f_{i}^{[l]}|\HKZ_{1}(X))$ or $\E(f_{k}^{[l]}|\HKZ_{1}(Y))$ converges to zero in probability for $[l]$-u.a.e.\ $x,y$, so
\[
\int \otimes_{i=0}^{k}f_{i}^{[l]} \dif \tilde\m_{x,y}
\to 0 \quad \text{as} \quad j\to\infty
\]
for $[l]$-u.a.e. $x_{0},\dots,x_{k}$ since $\ker\psi$ has full projections on coordinates.
\end{proof}
\begin{corollary}
\label{cor:m-ergodic}
For $[l]$-u.a.e.\ $x_{0},\dots,x_{k}$ the measure $\m_{x_{0},\dots,x_{k}}=\tilde\m_{x,y}$ is ergodic.
\end{corollary}
Note that even for a non-ergodic invariant measure on a regular system there may exist generic points, so the mere fact that $\vec x$ is generic for $\m_{\vec x}$ does not suffice.
\begin{proof}
In order to see that $\m_{x_{0},\dots,x_{k}}$ is ergodic it suffices to verify that for any continuous functions $f_{i,\epsilon}\in C(X_{i})$ we have
\begin{equation}
\label{eq:criterion-ergodicity-of-m}
\lim_{N}\aveFN \otimes_{i=0}^{k}f_{i}^{[l]}(T^{n}\vec\xi) = \int \otimes_{i=0}^{k}f_{i}^{[l]} \dif\m_{x_{0},\dots,x_{k}}
\quad \text{for } \m_{x_{0},\dots,x_{k}}\text{-a.e.\ }\vec\xi.
\end{equation}
Recall that for $[l]$-u.a.e.\ $x_{0},\dots,x_{k}$ the limit on the left-hand side of \eqref{eq:criterion-ergodicity-of-m} exists for $\m_{x_{0},\dots,x_{k}}$-a.e.\ $\vec\xi$ by Lemma~\ref{lem:uae-fiber-ae} and equals $\int \otimes_{i=0}^{k}f_{i}^{[l]} \dif\m_{\vec\xi}$.
Splitting the $f_{i,\epsilon}$'s as before it suffices to verify \eqref{eq:criterion-ergodicity-of-m} for the main terms, and this follows directly from \eqref{eq:m-erg-main-term}.
\end{proof}

\subsection{The sufficient special case of convergence to zero}
The last hypothesis of Proposition~\ref{prop:BFKO} is a certain special case of its conclusion.
Recall that we already have u.a.e.\ convergence to zero on $X^{2}$ (as defined in \eqref{eq:X}), but not yet in the required sense.
This is now corrected using Lemma~\ref{lem:cube-change-order}.
\begin{lemma}[Change of order in the cube construction]
\label{lem:cube-order}
Let $l\in\N$ and $P$ be a statement about points of $\prod_{i=0}^{k}X_{i}^{[l+1]}$.
Assume that for $[l+1]$-u.a.e.\ $x_{0},\dots,x_{k}$
we have $P(x_{0},\dots,x_{k})$.

Then for $[l]$-u.a.e.\ $x_{0},\dots,x_{k}$,
for $\m_{x_{0},\dots,x_{k}}$-a.e. $x'$,
we have $P(x_{0},\dots,x_{k},x')$.
\end{lemma}
Strictly speaking, the coordinates of $x'$ in $(x_{0},\dots,x_{k},x')$ should be attached to $x_{0},\dots,x_{k}$ but we do not want to introduce additional notation at this point.
\begin{proof}
The base case $k=0$ follows directly from \eqref{eq:motimesm}.

Assume now that $k>0$.
By the inductive hypothesis of this section the conclusion holds for $k-1$, so
for $[l]$-u.a.e.\ $x_{0},\dots,x_{k-1}$,
\emph{for $\m_{x_{0},\dots,x_{k-1}}$-a.e. $x'$,
for every $y_{k}\in Y_{k,l+1}$ and $\m_{y_{k}}$-a.e. $x_{k}$,}
we have $P(x_{0},\dots,x_{k-1},x',x_{k})$.

Using \eqref{eq:motimesm} we can rewrite the emphasized part of the statement as
``for $\m_{x_{0},\dots,x_{k-1}}$-a.e. $x'$,
for every $\tilde y_{k}\in Y_{k,l}$,
for every ergodic component $\mu_{e}$ of $(\m_{\tilde y_{k}})^{2}$ from a fixed full measure set, for $\mu_{e}$-a.e. $x_{k}$''
The conclusion follows by Lemma~\ref{lem:cube-change-order} and Lemma~\ref{lem:m-tilde-m}.
\end{proof}
\begin{lemma}
\label{lem:cube-meas}
Let $l,l'\in\N$ and assume \ref{eq:cf}$(k,l+l')$.
Then for $[l]$-u.a.e. $\vec x_{0}=(x_{0},\dots,x_{k})$,
for $\m_{\vec x_{0}}$-a.e. $\vec x_{1}$, \ldots, for $\m_{\vec x_{0},\dots,\vec x_{l'-1}}$-a.e. $\vec x_{l'}$
the ergodic averages of the function $\otimes_{i=0}^{k}f_{i}^{[l+l']}$ converge to zero at $(\vec x_{0},\dots,\vec x_{l'})$.
\end{lemma}
Again, the tensor product $\otimes_{i=0}^{k}f_{i}^{[l+l']}$ should be arranged in a different order, but in our opinion the above notation makes our goal more clear: it is not the function but the order in which we build the product space that changes.
\begin{proof}
We use induction on $l'$.
The case $l'=0$ is precisely Theorem~\ref{thm:RTT-cube}.
Assume that the conclusion is known for $l+1$ and $l'-1$.
The claim for $l$ and $l'$ follows by Lemma~\ref{lem:cube-order}.
\end{proof}
\begin{corollary}
\label{cor:orth}
Let $l,l'\in\N$ and assume \ref{eq:cf}$(k,l+l')$.
Then for $[l]$-u.a.e. $x_{0},\dots,x_{k}$
we have
$f_{0}^{[l]}\otimes\dots\otimes f_{k}^{[l]} \perp\HKZ_{l'}(\m_{x_{0},\dots,x_{k}})$.
\end{corollary}
\begin{proof}
This follows from Lemma~\ref{lem:cube-meas} by Lemma~\ref{lem:uae-fiber-ae}, the definition of cube measures \eqref{eq:mul+1},
the characterization of uniformity seminorms \eqref{eq:uniformity-seminorm-integral} and the ergodic theorem.
\end{proof}
\begin{proof}[Proof of Theorem~\ref{thm:RTT-cube} for $k+1$]
Let $k,l\in\N$ be fixed, our objective is to prove \RTT{k+1,l}.
Assume first \ref{eq:cf}$(k,l+1)$.
Then Lemma~\ref{lem:cube-meas} with $l'=1$ states that
for $[l]$-u.a.e. $x=(x_{0},\dots,x_{k})$,
for $\m_{x_{0},\dots,x_{k}}$-a.e. $x'$,
for any $f_{i,\epsilon}\in D_{i}$
we have
\[
\lim_{N}\aveFN \otimes_{i=0}^{k}f_{i}^{[l]}((\otimes_{i=0}^{k}T_{i}^{[l]})^{n}x)
\cdot \otimes_{i=0}^{k}f_{i}^{[l]}((\otimes_{i=0}^{k}T_{i}^{[l]})^{n}x')
= 0.
\]
For $[l]$-u.a.e. $x_{0},\dots,x_{k}$ we obtain genericity w.r.t.\ $\m_{x_{0},\dots,x_{k}}$ by Theorem~\ref{thm:return-times-disintegration}, ergodicity of $\m_{x_{0},\dots,x_{k}}$ by Corollary~\ref{cor:m-ergodic} and orthogonality of $\otimes_{i=0}^{k} f_{i}^{[l]}$ to the Kronecker factor of $\m_{x_{0},\dots,x_{k}}$ by Corollary~\ref{cor:orth}, so Proposition~\ref{prop:BFKO} with $X=(X_{0}^{[l]}\times\dots\times X_{k}^{[l]},\m_{x_{0},\dots,x_{k}})$ and $Y=(X_{k+1}^{[l]},\m_{y_{k+1}})$ implies the claimed convergence to zero $[l]$-u.a.e.

This takes care of the terms $f_{i,\epsilon,\perp}$ in the splittings $f_{i,\epsilon} = f_{i,\epsilon,\perp}+f_{i,\epsilon,\HKZ,j}+f_{i,\epsilon,err,j}$ given by \ref{eq:dec}$(k+l+1-i)$ (resp. $k+l$ for $i=0$).
By an approximation argument like in the proof of Lemma~\ref{lem:m-tilde-m} it suffices to consider the main terms, so we may assume that $\prod_{i=0}^{k} f_{i}^{[l]}((T_{i}^{[l]})^{n}x)$ is a nilsequence.
The claimed convergence a.e. in $x_{k+1}$ then follows from Theorem~\ref{thm:WW}.

Finally, assume \ref{eq:cf}$(k+1,l)$.
This means that we have either \ref{eq:cf}$(k,l+1)$ or $f_{k+1,\epsilon}\perp\HKZ_{l+1}(X_{k+1})$ for some $\epsilon$.
In the former case the limit is zero $[l]$-u.a.e.\ by the above argument and in the latter case by definition of $Y_{k+1,l}$ and Lemma~\ref{lem:Z1-char-weight}.
\end{proof}

\section{Wiener-Wintner return times theorem for nilsequences}
\label{sec:ww}
The first step in the proof is the identification of characteristic factors in the spirit of \cite{MR1357765}*{\textsection 4}.
\begin{lemma}
\label{lem:HKZ-char-for-WWRTT}
Let $f_{i}\in L^{\infty}(X_{i})$, $i=0,\dots,k$, and assume \ref{eq:cf}$(k,l)$.
Then for u.a.e.\ $x_{0},\dots,x_{k}$ and every $l$-step nilsequence $(a_{n})$ we have
\begin{equation}
\label{eq:ave-uniform}
\lim_{N\to\infty} \aveFN a_{n} \prod_{i=0}^{k}f_{i}(T_{i}^{n}x_{i}) = 0.
\end{equation}
\end{lemma}
One can formulate a uniform version of this result along the lines of the uniform Wiener-Wintner theorem \cite{arxiv:1208.3977}*{Theorem 4.1} but this would require additional notation.
\begin{proof}
By Corollary~\ref{cor:orth} we have $\otimes_{i=0}^{k}f_{i} \perp \HKZ_{l}(\m_{\vec x})$ for u.a.e.\ $\vec x\in X_{0}\times\dots\times X_{k}$ and by Theorem~\ref{thm:return-times-disintegration} u.a.e.\ $\vec x$ is fully generic for $\otimes_{i} f_{i}$ w.r.t.\ $\m_{\vec x}$.
The claim follows by Theorem~\ref{thm:WW}.
\end{proof}
Theorem~\ref{thm:WWRTT} now follows from equidistribution results on nilmanifolds.
\begin{proof}[Proof of Theorem~\ref{thm:WWRTT}]
Fix $k,l\in\N$.
By Lemma~\ref{lem:HKZ-char-for-WWRTT} it suffices to consider $f_{i}\in L^{\infty}(\HKZ_{l+k+1-i}(X_{i}))$.
By the pointwise ergodic theorem we can assume that each $f_{i}$ is a continuous function on a nilfactor of $X_{i}$.
The conclusion follows from equidistribution results on nilmanifolds \cite{MR2122919}*{Theorem B}.
\end{proof}

\bibliography{pzorin-ergodic-MR,pzorin-ergodic-preprints}
\end{document}